\documentclass[12pt]{amsart}
\usepackage{mathrsfs} 
\usepackage[sorted-cites]{amsrefs}
\usepackage[matrix,arrow,tips,curve]{xy}
\usepackage[english]{babel}
\usepackage{amsmath}
\usepackage{amssymb}
\numberwithin{equation}{section}
\theoremstyle{plain}
\usepackage{enumerate}
\usepackage{graphicx}
\usepackage{hyperref}
\usepackage{color}
\DeclareMathAlphabet{\pazocal}{OMS}{zplm}{m}{n}

\usepackage{xcolor}

\newtheorem{theorem}{Theorem} [section]
\newtheorem{lemma}{Lemma}[section]
\newtheorem{proposition}{Proposition}[section]

\newtheorem{example}{Example}
\theoremstyle{remark}
\newtheorem{remark}{Remark}[section]

\hypersetup{pdfpagemode=UseNone} 
\hypersetup{pdfstartview=FitH}

\title[]{Eigenvalue Estimates for Schrödinger Operators on Ricci Shrinkers}

\author[Cheng]{Xu Cheng}
\address[Xu Cheng]{Instituto de Matem\'atica e Estat\'istica, Universidade Federal Fluminense, S\~ao Domingos,
Niter\'oi, RJ 24210-201, Brazil}
\email{xucheng@id.uff.br}
\author[Conrado]{Franciele Conrado}\address[Franciele Conrado]{Departamento de Matem\'atica, Universidade Federal de Sergipe, Jardim Rosa Elze, S\~ao Crist\'ov\~ao, SE  49100-000, Brazil}
\email{franciele@mat.ufs.br}
\author[Pinheiro]{Neilha Pinheiro}
\address[Neilha Pinheiro]{Universidade Federal do Amazonas\\
Instituto de Ci\^encias Exatas\\
Manaus, AM 69080900, Brazil.}
\email{neilha@ufam.edu.br}

\author[Zhou]{Detang Zhou}
\address[Detang Zhou]{ Instituto de Matem\'atica e Estat\'istica, Universidade Federal Fluminense, S\~ao Domingos,
Niter\'oi, RJ 24210-201, Brazil}
\email{zhoud@id.uff.br}
\date{}
\thanks{X. Cheng was partially supported by CNPq/Brazil [Grant:  314504/2023-0]. }
\thanks{D. Zhou was partially supported by CNPq/Brazil [Grant:  308067/2023-1] and FAPERJ/ Brazil [Grant: E-26/200.386/2023].}

\thanks{F. Conrado was partially supported by CNPq/Brazil [Grant: 408834/2023-4].}

\begin{document}


\begin{abstract}
 Let  $(M, g, f, \tau)$ be a complete Ricci shrinker satisfying  $\textrm{Ric}+\nabla^2f=\frac{g}{2\tau}$  and let $R$ denote its scalar curvature. For  a confined function $V$ on $M$, we obtain a lower bound for the lowest eigenvalue  of the Schr\"odinger operator
\(
-\Delta+\frac{R}{4}+V,
\)
expressed in terms of an integral quantity involving $V$ and the shrinker entropy, and  the equality case is characterized  by the potential functions. 
We further generalize this estimate to  complete Riemannian manifolds via Perelman’s $\mu$-functional. We also study the drifted Schr\"odinger operator $-\Delta_f+V$ on smooth metric measure spaces. In particular, on Ricci shrinkers, we derive a lower bound for its lowest eigenvalue, with equality if and only if  $V$ is  affine. 
\end{abstract}
\maketitle

\section{Introduction}

\medskip

 Let $(M,g)$ be an $n$-dimensional Riemannian manifold. For $f\in C^\infty(M)$ and a constant $\tau>0$, the quadruple $(M,g,f,\tau)$ is called a \textit{shrinking gradient Ricci soliton} or simply a \textit{Ricci shrinker} if 
\begin{equation}\label{shrinker}
   \textrm{Ric}+\nabla^2f= \frac{1}{2\tau} g,
\end{equation}
 where $\nabla^2 f$ denotes the Hessian of $f$ and $\textrm{Ric}$ is the Ricci curvature tensor of $(M,g)$.  In this case,  $f$ is called a \textit{potential function} of the Ricci shrinker.

Any Einstein manifold of positive scalar curvature is a Ricci shrinker with constant potential function. An important example of a non-Einstein Ricci shrinker is the {\it Gaussian shrinker}, given by Euclidean space $\mathbb{R}^n$ with potential function  $f(x) = \frac{|x|^2}{4\tau}$, $x\in\mathbb{R}^n$. There are many other examples; see, e.g., \cite{Cao2010}. Ricci shrinkers, lying at the intersection of critical metrics and geometric flows, have been extensively studied; see, for example, \cites{Cao2010,caozhou,RF1,chengzhou, hamilton},  among others.

In this paper, we study the spectrum properties of the Schr\"odinger operators. Let $(M,g)$ be a complete Riemannian manifold. Given $V\in L^{\infty}_{\textrm{loc}}(M)$ real-valued and bounded below, it is known that the Schr\"odinger operator $-\Delta+V$  is
 essentially self-adjoint   on  $L^2(M)$. Moreover, if $M$ is compact or if $M$ is  non-compact and  $V$  is confining in the sense that
\begin{equation}\label{confin}
   \lim_{d(x,p_0)\to\infty}V(x)=+\infty, \  \textrm{ for some } p_0\in M, 
\end{equation}
 then  $-\Delta+V$ on $L^2(M)$ has discrete spectrum; see, e.g., \cite{L}.

For $-\Delta+V$  on $L^2(\mathbb{R}^n)$,  Keller \cite{Keller} proposed the problem of minimizing its lowest
eigenvalue among potentials $V$ with prescribed $L^p$ norm.  Lieb and Thirring \cite{LiebThirring}  solved this problem in terms of optimal constants in certain Sobolev interpolation inequalities.
More recently, Frank \cite{Frank} studied the corresponding minimization for confining potential $V$  satisfying $e^{-4\tau V}\in L^1(\mathbb{R}^n)$ for some $\tau>0$. He proved that, among all such potentials  with   $\int_{\mathbb{R}^n}e^{-4\tau V(x)}dx $ fixed, the lowest eigenvalue is minimized,  when $4\tau V$ is a harmonic oscillator potential. More precisely, 

\begin{theorem}[Frank, \cite{Frank}]\label{Frank}
Let  $V:\mathbb{R}^n\to\mathbb{R}$ be a confining potential such that $e^{-4\tau V}\in L^1(\mathbb{R}^n)$ for some $\tau>0$. Then, the lowest eigenvalue $E_0$ of  $ -\Delta+V$ satisfies
\begin{equation}\label{frankestimate}
   E_0 \geq -\frac{1}{4\tau}\ln{\int_{\mathbb{R}^ n}}e^{-4\tau V(x)}dx+\frac{n}{4\tau}\Big(1+\frac{1}{2}\ln{(4\pi\tau)}\Big).
\end{equation}
Moreover, the equality holds if and only if 
$$4\tau V(x)=\frac{|x-b|^2}{4\tau}+\textrm{const},$$
for some $b\in\mathbb{R}^n.$
\end{theorem}

Motivated by the fact that harmonic oscillators on $\mathbb{R}^n$ correspond to the Gaussian shrinker in Ricci flow, we study the setting of Ricci shrinkers and subsequently derive lower bounds  for the lowest eigenvalues of two  Schrödinger-type operators on Ricci shrinkers  and, more generally, on complete Riemannian manifolds.  The first operator is 
$$ -\Delta+\frac{R}{4}+V,$$
acting on $L^2(M)$, where $R$ denotes the scalar curvature of $(M,g)$ and $V\in L^{\infty}_{\textrm{loc}}(M)$ is real-valued and bounded below. 
We  prove the following result.

\begin{theorem}\label{mainteo1}Let $(M,g, f,\tau)$ be an $n$-dimensional complete Ricci shrinker, and let $V$ satisfy  $e^{-4\tau V}\in L^1(M)$. If $M$ is compact, or if $M$ is non-compact and $V$ is  confining, then 
\[
\mathcal{L}=-\Delta+\frac{R}{4}+V
\]
on $L^2(M)$ has discrete spectrum, and its lowest
eigenvalue satisfies
\begin{equation}\label{e1}
    \lambda_0(\mathcal{L})\geq -\frac{1}{4\tau}\ln\int_M e^{-4\tau V}\, dv + \frac{\mu_s}{4\tau}+\frac{n}{4\tau}\left(1+\frac{1}{2}\ln(4\pi\tau)\right),
\end{equation}
where  $\mu_s$ denotes the entroy of Ricci shrinker given by \eqref{rentropy}.

Moreover, equality in \eqref{e1} holds if and only if  $4\tau V$  is a potential function of the  Ricci shrinker.

\end{theorem}
\begin{remark}
Theorem \ref{mainteo1} implies that, among all potentials $V$ such that   $\int_{M}e^{-4\tau V}d\upsilon $ is fixed, the lowest eigenvalue is minimized,  when $4\tau V$  is a potential function of the  Ricci shrinker.
\end{remark}

Next we generalize  Theorem \ref{mainteo1} to  Riemannian manifolds.  More precisely, we prove the following result.

\begin{theorem}\label{mainteo2}
Let $(M,g)$ be an $n$-dimensional complete  Riemannian manifold whose scalar curvature $R$ is bounded below, and let $V$ satisfy $e^{-4\tau V}\in L^1(M)$  for some $\tau>0$. If $M$ is compact, or if $M$ is non-compact and $V$ is  confining, then 
\[
\mathcal{L}=-\Delta+\frac{R}{4}+V
\]
on $L^2(M)$ has discrete spectrum, and its lowest
eigenvalue satisfies
    
\begin{equation}\label{lowbound}
 \lambda_0(\mathcal{L})\geq -\frac{1}{4\tau}\ln\int_M e^{-4\tau V}\, dv + \frac{\mu(g,\tau)}{4\tau}+\frac{n}{4\tau}\left(1+\frac{1}{2}\ln(4\pi\tau)\right),  
\end{equation}
where $\mu(g,\tau)$ is Perelman’s $\mu$-functional defined by \eqref{mu}.

Moreover, equality in \eqref{lowbound} holds if and only if  $\mu(g,\tau)$ is attained by $4\tau V+C$, where $C=\ln\left( (4\pi\tau)^{-\frac{n}{2}}\int_M e^{-4\tau V} \, d\upsilon\right).$

\end{theorem}

Observe that  \eqref{e1} can be recovered from Theorem \ref{mainteo2} by using the fact that on a  Ricci shrinker one has $\mu(g,\tau)=\mu_s$.

\vskip 2mm

The second operator we consider is the drifted Schrödinger operator on a  complete smooth metric measure space $(M,g,f)$:
$$-\Delta_f+V,$$
acting on $L^2(M,e^{-f}d\upsilon)$, where $V\in L^{\infty}_{\textrm{loc}}(M)$ is real-valued and bounded below. 
Recall that   the {\it drifted Laplacian} is defined by
$$\Delta_f=\Delta-\langle \nabla f, \nabla \cdot \rangle.$$
The operator $\Delta_f$ is  essentially self-adjoint  on  $L^2(M,e^{-f}d\upsilon)$.
Its spectrum properties have been studied in various settings; see, for example, Cheng and Zhou \cite{chengzhou}. For $-\Delta_f+V$, we obtain  the following result on Ricci shrinkers.

\begin{theorem}\label{mainteo3}
Let $(M,g,f,\tau)$ be an $n$-dimensional complete Ricci shrinker such that  $d\mu=(4\pi\tau)^{-\frac{n}{2}}e^{-f}d\upsilon$ is a probability measure, and let $V$  satisfy $e^{-4\tau V}\in L^1(M,d\mu)$. If $M$ is compact, or if $M$ is non-compact and $V$ is  confining, then 
$$\mathcal{L}_f=-\Delta_f+V$$ 
on $L^2(M,d\mu)$ has discrete spectrum, and its lowest eigenvalue satisfies
    \begin{equation}\label{fzestimate3}\lambda_0(\mathcal{L}_f)\geq -\frac{1}{4\tau}\ln\int_M e^{-4\tau V}\, d\mu.\end{equation}
Moreover, equality holds if and only if $V\in C^{\infty}(M)$ with $\nabla^2V=0$.
\end{theorem} 
\begin{remark}\label{remav} The rigidity in Theorem \ref{mainteo3} can be formulated more explicitly. Observe that, $V\in C^{\infty}(M)$ with $\nabla^2V=0$ if and only if one of the following alternatives holds:
\begin{enumerate}
    \item  $V$ is a constant function, or
    \item $M$ is isometric to $\Sigma\times \mathbb{R}^k$, where $\Sigma$ is an $(n-k)$-dimensional complete Riemannian manifold, and by passing an isometry, for $(x,y)\in \Sigma\times \mathbb{R}^k$, $V$ satisfies
 $$V(x,y)=\frac{1}{4\tau}\langle y,b\rangle + \textrm{constant}, \ \text{ for some } b\in \mathbb{R}^k. $$
\end{enumerate}
In particular, see Theorem \ref{t22v2} for the compact case.
\end{remark}
\begin{remark}
  Compared with Theorem \ref{mainteo1}, Theorem \ref{mainteo3} provides a  lower bound depending only on $V$. Its rigidity statement  reflects the affinity of the optimal potentials, characterized by  $\nabla^2 V=0$.
\end{remark}

Finally, we give a lower bound for the bottom of spectrum of $-\Delta_f+V$  in terms of Perelman's $\mu$-functional on the smooth metric measure space $(M,g,f)$; see Theorem \ref{mainteo4}.

\begin{remark}
Theorems \ref{mainteo1}, \ref{mainteo2}, \ref{mainteo3}, and \ref{mainteo4} may be viewed as  illustrations of  the interaction between Schr\"odinger operators and Perelman’s entropy functional.

\end{remark}

The rest of the paper is organized as follows. In Section 2, we fix notation and recall some definitions and results used throughout the paper. In Section 3, we prove the theorems.

\section{Preliminaries}

\subsection{Smooth metric measure spaces}\hfill\break\vspace{-3mm}

A \textit{complete smooth metric measure space} $(M, g, f)$  is a complete connected $n$-dimensional Riemannian manifold $(M,g)$ equipped with a smooth measure $e^{-f}d\upsilon$, where $f$ is a smooth function on $M$ and $d\upsilon$ denotes the Riemannian volume measure associated with $g$. 

\vskip 2mm
If  the \textit{Bakry-\'Emery Ricci curvature} of $(M,g,f)$ satisfies
$$\textrm{Ric}_f=\textrm{Ric}+\nabla^2f\geq \frac{1}{2\tau} g,$$
 for some constant $\tau>0$, then it follows from \cite{Morgan} that  the weighted volume $\int_M e^{-f} \, d\upsilon $ is finite. Hence,  after normalizing $\int_M(4\pi\tau)^{-\frac{n}{2}}e^{-f}d\upsilon=1$, the logarithmic Sobolev inequality of Bakry-Émery \cite{BE} holds on $M$:
\begin{equation}\label{log}
  \int_M u^2 \ln u^2 d\mu\leq 4\tau \int_M |\nabla u|^2 \, d\mu,  
\end{equation}
for all   $u\in H^1(M,d\mu)$ with $\int_M u^2 \, d\mu=1$, where  $d\mu=(4\pi\tau)^{-\frac{n}{2}}e^{-f}d\upsilon$.

Recently, Conrado \cite{Conrado} studied the associated rigidity problem in \eqref{log} and proved that 
\begin{proposition}\label{lsi}
   If the equality in \eqref{log} holds for some function $u$, then $u\in C^{\infty}(M)$ and $$\nabla^2 (\ln u^2)=0.$$ 
  
\end{proposition}
    
 We study two Schrödinger-type operators:$$-\Delta+\frac{R}{4}+V\ \text{ and } \ -\Delta_f+V,$$ where $R$ denotes the scalar curvature of $M$.  Throughout the paper, we assume  that  $V\in L^{\infty}_{\textrm{loc}}(M)$ is real-valued and bounded below, unless otherwise specified.
 \begin{remark}
      For a complete noncompact manifold \(M\), we call a function \(V\) confining if it satisfies \eqref{confin}. This condition is imposed only to ensure that \(-\Delta+V\) has compact resolvent and hence discrete spectrum. Other notions of confining potentials are also possible, provided that they imply that \(-\Delta+V\) has compact resolvent. 
 \end{remark}

\subsection{Perelman's $\mathcal{W}$-functional}\label{subsec2} \hfill\break\vspace{-3mm}

In this subsection, we recall some facts about Perelman's $\mathcal{W}$-functional. For more details, see \cite{RF1} and also \cite{RFT}.
\vskip 2mm

Let $M$ be an $n$-dimensional smooth manifold. Denote by $\mathcal{R}(M)$ the space of all Riemannian metrics on $M$.  {\it Perelman's $\mathcal{W}$-functional} is defined  by
$$
\mathcal{W}(g,f,\tau)=\int_M [\tau (R+|\nabla f|^2)+f-n](4\pi
\tau)^{-\frac{n}{2}}e^{-f}d\upsilon,
$$
where $g\in \mathcal{R}(M)$, $f\in C^{\infty}(M)$ and $\tau>0$.

\vskip 2mm

The $\mathcal{W}$-functional has the following  properties:

\begin{enumerate}
\item[(i)](Scale invariance) 
$\mathcal{W}(cg, f, c\tau)=\mathcal{W}(g, f, \tau), \textrm{ for all } c>0.$
\item [(ii)](Diffeomorphism invariance) 
$$\mathcal{W}(g, f, \tau)=\mathcal{W}(\Phi^*g, \Phi^*f, \tau), \textrm{ for any diffeomorphism } \Phi.$$
\end{enumerate}
Define 
$$\Omega(M)=\left\{(g,f,\tau)\in \mathcal{R}(M)\times C^{\infty}(M)\times \mathbb{R}^+; \ \int_M(4\pi\tau)^{-\frac{n}{2}}e^{-f}d\upsilon_g=1\right\}.$$
We consider the restriction of $\mathcal{W}$ to $\Omega(M)$. The associated {\it $\mu$-functional} and {\it $\nu$-functional}  are defined by
\begin{equation}\label{mu}\mu(g,\tau)=\inf \left\{\mathcal{W}(g,f,\tau);  \ (g,f,\tau)\in\Omega(M)\right\},
\end{equation}
\begin{equation}\label{nu}
    \nu(g)=\inf\left\{\mu(g,\tau); \ \tau>0\right\}.
\end{equation}
The functionals $\mu$ and $\nu$ inherit scale and diffeomorphism invariance properties.
\vskip 2mm
For $g\in \mathcal{R}(M)$, $\phi\in C^{\infty}(M)$ and $\tau>0$, we define
\begin{equation*}\label{perelman} 
  \mathcal{K}(g,\phi,\tau)=\int_M\left(4\tau|\nabla \phi|^2+\tau R\phi^2 - \phi^2\ln \phi^2\right) \, d\upsilon
-n\left(1+\frac{1}{2}\ln(4\pi\tau)\right).  
\end{equation*}
Then,
\begin{equation}\label{wk}
    \mathcal{W}(g,f,\tau)=\mathcal{K}(g,\phi,\tau), \text{ where } \phi^2=(4\pi\tau)^{-\frac{n}{2}}e^{-f}.
\end{equation}
Therefore
\begin{equation}\label{perel}
    \mu(g,\tau)=\inf \left\{\mathcal{K}(g,\phi,\tau); \ \phi\in C^{\infty}(M) \ \text{with} \ \int_M \phi^2 \, d\upsilon=1\right\}.
\end{equation}

The next lemma characterizes the critical points of \eqref{mu}, and in particular its minimizers.
\begin{lemma}\label{el} The Euler-Lagrange equation of (\ref{mu}) is
\begin{equation}\label{eullag}
    \tau(R+2\Delta f-|\nabla f|^2)+f-n= \textrm{const}.
\end{equation}
If $f$ is a minimizer of (\ref{mu}), then
\begin{equation}\label{min}
 \tau(R+2\Delta f-|\nabla f|^2)+f-n=\mu(g,\tau).   
\end{equation}
\end{lemma}

\medskip

Note that we do not assume that the infima $\mu(g,\tau)$ and $\nu(g)$ are finite. But they are finite for Ricci shrinkers; see Theorem \ref{tmrs}.  

\medskip

On compact smooth manifolds, the infimum $\mu(g,\tau)$ is finite and is attained by a smooth function, for any given $g$ and $\tau$.  
On the other hand, there exist compact Riemannian manifolds such that $\nu(g)=-\infty$, as is the case for scalar-flat manifolds.

\subsection{Ricci shrinkers} \hfill\break\vspace{-3mm}

In this subsection, let $(M,g, f,\tau)$ be a Ricci shrinker satisfying \eqref{shrinker}. Then
\begin{equation}\label{eqs}
    \Delta f+R=\frac{n}{2\tau}.
\end{equation}
Moreover, by a result of Hamilton \cite{hamilton}, one has
\begin{equation}\label{hamilton}\tau(R+|\nabla f|^2)-f=C,
\end{equation}
where  $C$ is constant on $M$.

\medskip

By a result of Chen \cite{chen} on complete ancient solutions, a Ricci shrinker has nonnegative scalar curvature (cf. Proposition 5.5 in \cite{Cao2010}). Furthermore, using the maximum principle, Pigola, Rimoldi, and Setti \cite{prs} proved that a Ricci shrinker has positive scalar curvature unless it is isometric to Euclidean space.

In \cite{caozhou}, Cao and Zhou proved the following growth estimate for the potential functions of noncompact Ricci shrinkers.

\begin{theorem}[Cao-Zhou, \cite{caozhou}]\label{CZ1} Let $(M,g,f,\tau)$ be an $n$-dimensional complete noncompact normalized Ricci shrinker such that $f=|\nabla f|^2+R$ and  $\tau=1$. Fix a point $p\in M$ and consider the distance function $r(x)=d(x,p)$. Then
$$\frac{1}{4}(r(x)-c)^2\leq f(x)\leq \frac{1}{4}(r(x)+c)^2$$
\noindent for all $x\in M$ with $r(x)$ sufficiently large, where $c$ is a positive constant depending only on $n$ and the geometry of $g$ on the unit ball $B_1(p)$.
\end{theorem}

\begin{remark}\label{rCZ1} By Theorem \ref{CZ1}, the potential function of a noncompact Ricci shrinker satisfies that $f(x) \to +\infty$ as $r(x)\to +\infty$. 
\end{remark}

  For Ricci shrinkers,  the functionals $\mu$ and $\nu$ have  the following properties.

\begin{theorem}\label{tmrs}  Let  $(M,g,f,\tau)$ be a  Ricci shrinker such that $(g,f,\tau)\in\Omega(M)$. Then $\nu(g)>-\infty$. Moreover,
\begin{equation}
   \nu(g)=\mu(g,\tau)=\mathcal{W}(g,f,\tau),
\end{equation}
and
\begin{eqnarray}\label{mu1a}
    \mu(g,\tau)
    &=&f-\tau(R+|\nabla f|^2).
\end{eqnarray}
 \end{theorem}
 
\begin{remark}
    Equation \eqref{mu1a} follows directly from \eqref{min} and \eqref{eqs}.
\end{remark}

For a general Ricci shrinker $(M,g,f,\tau)$,  
we have the finite {\it $f$-volume}
$$V_f(M)=\int_M (4\pi\tau)^{-\frac{n}{2}}e^{-f}d\upsilon.$$
 and define 
\begin{equation}\label{rentropy}
     \mu_{s}:=\ln V_f(M)+f-\tau(R+|\nabla f|^2).
\end{equation}
By  \eqref{hamilton},  $\mu_s$ is constant on $M$.
In particular,  $V_f(M)=1$ if and only if
\begin{equation}\label{rentropy1}
 \mu_s=f-\tau(R+|\nabla f|^2).
\end{equation}
Note that, by \eqref{rentropy},  $\mu_s$ is invariant under translations of the potential function $f$. Hence  by \eqref{mu1a} and \eqref{rentropy1}, $\mu_s$ is the (Perelman) entropy of the Ricci shrinker, that is, 
\begin{equation}
   \mu_s=\mu(g,\tau)=\nu(g). 
\end{equation}

\begin{example}
  The entropy of the Gaussian shrinker is equal to zero.
\end{example}

\begin{example}
Let $M$ be an $n$-dimensional  Einstein manifold of positive scalar curvature $R$. In this case, we have
$$\mu_s=\ln V(M)-\frac{n}{2}\left(1+\ln\left(\frac{2\pi n}{R}\right)\right),$$\noindent where $V(M)$ denotes the volume of $M$. 
\end{example}

\begin{example}
The entropy of the cylinder shrinker $\mathbb{S}^{n-k}(r) \times \mathbb{R}^k$ is given by 
$$\mu_{s}=\log \omega_{n-k}-\frac{(n-k)}{2}\left[\log\left(\frac{2\pi}{n-k-1}\right)+1\right],$$
\noindent where $\omega_{n-k}$ denotes the area of the unit $(n-k)$-dimensional sphere.
\end{example}

\section{ Proofs of the main results}

We begin with the following lemma.

\begin{lemma}\label{rlow}
Let $(M,g,f, \tau)$ be an $n$-dimensional complete  Ricci shrinker. Then, 
$$\mathcal{L}=-\Delta+\frac{R}{4}+\frac{f}{4\tau},$$
\noindent  has discrete spectrum and its lowest eigenvalue is given  by
$$\lambda_0(\mathcal{L})=-\frac{1}{4\tau}\ln\int_M e^{-f }\, dv + \frac{\mu_s}{4\tau}+\frac{n}{4\tau}\left(1+\frac{1}{2}\ln(4\pi\tau)\right),$$
\noindent where $\mu_s$ denotes the entropy of the Ricci shrinker.
    
\end{lemma}
\begin{proof}
   In the  noncompact case, by the nonnegativity of $R$ and Remark \ref{rCZ1},  the function $$\frac{R}{4}+\frac{f}{4\tau}$$ \noindent is a confining potential. Therefore, $\mathcal{L}$ has discrete spectrum.
Next, we compute the lowest eigenvalue of $\mathcal{L}$. 
 It follows from \eqref{rentropy} and \eqref{eqs} that
\begin{equation}\label{rf}
\begin{aligned}
    \frac14\left( R+\frac{f}{\tau}\right)
    &=\frac14\left(\frac{n}{\tau}+\frac1{\tau}\mu_s-\frac1{\tau}\ln V_f(M)-2\Delta f+|\nabla f|^2\right) \\
    &=\frac{|\nabla f|^2}{4}-\frac{\Delta f}{2} +\lambda,
    \end{aligned}
\end{equation}
 where
$$\lambda:= \frac{\mu_s}{4\tau}+\frac{n}{4\tau}-\frac{\ln V_f(M)}{4\tau}.$$
Let $U:L^2(M)\rightarrow L^2(M,e^{-f}d\upsilon)$ be  the unitary isomorphism  given by $U(u)=ue^{\frac{f}{2}}$. A direct computation shows that $$-\Delta_f=U\left(-\Delta-\frac{1}{2}\Delta f+\frac{1}{4}|\nabla f|^2\right)U^{-1}.$$
Thus, by \eqref{rf}, we have
$$-\Delta_f+\lambda= U\mathcal{L}U^{-1}.$$
Consequently,
\begin{eqnarray*}
\lambda_0(\mathcal{L}) & = & \lambda_0\left(-\Delta_f+\lambda \right)
= \lambda_0\left(-\Delta_f\right)+\lambda.\end{eqnarray*}
Since  $\int_M e^{-f} \, d\upsilon $ is finite, $\lambda_0\left(-\Delta_f\right)=0$. Therefore, $\lambda_0(\mathcal{L})=\lambda$, that is,
$$\lambda_0(\mathcal{L})=-\frac{1}{4\tau}\ln\int_M e^{-f}\, dv + \frac{\mu_s}{4\tau}+\frac{n}{4\tau}\left(1+\frac{1}{2}\ln(4\pi\tau)\right).$$
\end{proof}    
\vskip 2mm
A key ingredient in the proofs of Theorems \ref{mainteo1}, \ref{mainteo2}, \ref{mainteo3}, and \ref{mainteo4} is the following variational principle.

\begin{lemma}[Gibbs variational principle]\label{gibbs} Let $(M,g)$ be a Riemannian manifold, let $\mu$ be a measure on $M$, and let $H : M \to \mathbb{R}$ be a measurable function such that $e^{-tH}\in L^1(M,\mu)$, for some $t>0$. Then 

$$\inf_{\int_M u^2 \, d\mu=1} \left\{\int_M u^2\ln u^2 \, d\mu+t\int_M H u^2 \, d\mu\right\}=-\ln\int_M e^{-tH} \, d\mu.$$
Moreover, the infimum is attained uniquely at the Gibbs density
\[
u^2(x)=\frac{e^{-tH}}{\int_M e^{-tH} \, d\mu}.
\]
\end{lemma}
We will also use the following lemma.
\begin{lemma}\label{r1}
    Let $(M,g,f)$ be an $n$-dimensional complete smooth metric measure space with scalar curvature $R$ bounded below.  Fix a number $\tau>0$ and let $d\mu=(4\pi\tau)^{-\frac{n}{2}}e^{-f}d\upsilon$ be the weighted measure. Then  the spectra of $-\Delta+\frac{R}{4}+V$ on $L^2(M)$ and   $-\Delta_f+\frac{R_f}{4}+V$ on $L^2(M,\mu)$ coincide, where $R_f=R+|\nabla f|^2+2\Delta_ff$. In particular, the bottoms of their spectra satisfy
$$\lambda_0\left(-\Delta_f+\frac{R_f}{4}+V\right)= \lambda_0\left(-\Delta+\frac{R}{4}+V\right),$$
  
\end{lemma}
\begin{proof}
    As proved in Lemma \ref{rlow}, there exists a unitary isomorphism $U:L^2(M)\rightarrow L^2(M,\mu)$ such that
$$-\Delta_f=U\left(-\Delta-\frac{1}{2}\Delta f+\frac{1}{4}|\nabla f|^2\right)U^{-1}.$$
Using \(R_f=R+2\Delta f-|\nabla f|^2,\)
we obtain
$$-\Delta_f+\frac{R_f}{4}+V=U\left(-\Delta+\frac{R}{4}+V\right)U^{-1}$$
Therefore, the conclusion follows.
\end{proof}

\vskip 2mm

Now we prove Theorem \ref{mainteo1}.

\begin{proof}[Proof of Theorem  \ref{mainteo1}]  In the noncompact case,
since $R\geq 0$ and $V$ is confining,   $V+\frac{R}{4}$ is  a confining potential. Hence,  $\mathcal{L}$ has discrete spectrum. 

 By the  invariance of $\mu_s$ under translations of the potential functions, we may assume, up to a translation, that the measure $d\mu=(4\pi\tau)^{-\frac{n}{2}}e^{-f}d\upsilon$ is a probability measure. By Lemma \ref{r1}, $$\mathcal{L}_f=-\Delta_f+\frac{R_f}{4}+V$$ on $L^2(M,d\mu)$ has the same discrete spectrum, and hence
there exists a positive function  $u\in  C^{\infty}(M)$ with $\int_M u^2 \, d\mu=1$ such that $$\lambda_0(\mathcal{L}_f)=\int_M\left(|\nabla u|^2 +\frac{R_f}{4}u^2+Vu^2\right)d\mu.$$
 Using \eqref{eqs} and \eqref{rentropy1}, we obtain
 \begin{equation}\label{rf}
\begin{aligned}
R_f & =  R-|\nabla f|^2+2\Delta f\\
& = -R-|\nabla f|^2+\frac{n}{\tau}\\
& =  \frac{1}{\tau}\left(n+\mu_s-f\right).
\end{aligned}
\end{equation}
Therefore, 
\begin{equation}\label{e1t2}\lambda_0(\mathcal{L}_f)=\int_M|\nabla u|^2 \, d\mu + \int_M \left(V-\frac{f}{4\tau}\right) u^2 \, d\mu + \frac{1}{4\tau}(n+\mu_s).\end{equation}
By  \eqref{log} and then Lemma \ref{gibbs}, we have
\begin{equation}\label{e2t2}
  \begin{aligned}
      \lambda_0&(\mathcal{L}_f)\\
&\geq \frac{1}{4\tau}\left\{\int_M u^2\ln u^2 \, d\mu + 4\tau\int_M \left(V-\frac{f}{4\tau}\right) u^2 \, d\mu \right\} + \frac{n}{4\tau}+\frac{\mu_s}{4\tau}\\
& \geq - \frac{1}{4\tau} \ln \int_M e^{-4\tau \left(V-\frac{f}{4\tau}\right)}d\mu + \frac{n}{4\tau}+\frac{\mu_s}{4\tau}\\
& =  - \frac{1}{4\tau} \ln \int_M e^{-4\tau V} (4\pi\tau)^{-\frac{n}{2}}d\upsilon + \frac{n}{4\tau}+\frac{\mu_s}{4\tau}.
  \end{aligned}  
\end{equation}
Hence,
\begin{align}\label{e4t2}
    \lambda_0(\mathcal{L})&= \lambda_0(\mathcal{L}_f )\nonumber\\
    &\geq  -\frac{1}{4\tau}\ln\int_M e^{-4\tau V}\, dv + \frac{\mu_s}{4\tau}+\frac{n}{4\tau}\left(1+\frac{1}{2}\ln(4\pi\tau)\right),
\end{align}
which is \eqref{e1}.
Assume that  equality holds in (\ref{e4t2}). Then all inequalities in \eqref{e2t2} become equalities. In particular, by Proposition \ref{lsi},   $\nabla^2 (\ln u^2)=0$, and by  Lemma \ref{gibbs}, 
 $$u^2=\frac{e^{-4\tau V+f}}{(4\pi\tau)^{-\frac{n}{2}}\int_M e^{-4\tau V} \, d\upsilon}.$$
 Then,
$$4\tau V=f-\ln u^2+\textrm{const.}$$
Consequently,
$$\textrm{Ric}_{4\tau V}=\textrm{Ric}_f-\nabla^2 (\ln u^2)=\textrm{Ric}_f.$$
Therefore, we have $$\textrm{Ric}_{4\tau V}=\frac{1}{2\tau}g,$$
which shows that  $4\tau V$ is a potential
function of the Ricci shrinker.
\vskip 2mm
Conversely, if $4\tau V$ satisfies $$\textrm{Ric}_{4\tau V}=\frac{1}{2\tau}g, $$
then  equality holds in \eqref{e4t2} by Lemma \ref{rlow}.
\end{proof}

\vskip 2mm

Now we prove Theorem \ref{mainteo2}.

\begin{proof}[Proof of Theorem \ref{mainteo2}] By an argument analogous to that in the proof of Theorem \ref{mainteo1}, $\mathcal{L} $ has discrete spectrum, and hence there exists a positive  function $\phi\in  C^{\infty}(M)$ with  $\int_M \phi^2 \, d\upsilon=1$, such that 
\begin{eqnarray*} \lambda_0(\mathcal{L})  & = & \int_M\left(|\nabla \phi|^2+\frac{R}{4}\phi^2+V\phi^2\right) \, d\upsilon\\
& = & \int_M\left(\frac{1}{4}|\nabla \ln \phi^2|^2+\frac{R}{4}+V\right) \phi^2 \, d\upsilon.
\end{eqnarray*}
\noindent It follows from  \eqref{perel} that, 
\begin{equation}\label{mu1}
\begin{aligned}
\int_M\Big(\tau\big|\nabla\ln \phi^2\big|^2+\tau R - \ln \phi^2\Big)
\phi^2\,d\upsilon
&\\
\qquad
-n\left(1+\frac{1}{2}\ln(4\pi\tau)\right)
&\ge \mu(g,\tau).
\end{aligned}
\end{equation}
 Then, using \eqref{mu1} and  Lemma \ref{gibbs},   we obtain
\begin{equation}\label{teo3p}
    \begin{aligned}
\lambda_0(\mathcal{L})&=\int_M \left(\frac{1}{4}|\nabla \ln \phi^2|^2+\frac{R}{4}+V\right) \phi^2 \, d\upsilon\\
&\geq  \frac{1}{4\tau} \left(\int_M \phi^2 \ln \phi^2 \, d\upsilon+4\tau \int_M V\phi^2 \, d\upsilon\right)\\
&\qquad +\frac{\mu(g,\tau)}{4\tau}+\frac{n}{4\tau}\left(1+\frac{1}{2}\ln(4\pi\tau)\right)\\
&\geq -\frac{1}{4\tau}\ln\int_M e^{-4\tau V}\, dv+\frac{\mu(g,\tau)}{4\tau}+\frac{n}{4\tau}\left(1+\frac{1}{2}\ln(4\pi\tau)\right).
\end{aligned}
\end{equation}

\noindent Therefore we have proved \eqref{lowbound}.

If equality holds in \eqref{teo3p}, then all inequalities in the proof of \eqref{teo3p}
must be equalities.   Consequently, equality  in \eqref{mu1} implies that  $\mu(g,\tau)$ is attained by $\phi$, and  equality in Lemma \ref{gibbs} yields
$$\phi^2=\frac{e^{-4\tau V}}{\int_M e^{-4\tau V} \, d\upsilon}.$$
Let $f$  be defined by $\phi^2=(4\pi\tau)^{-\frac{n}{2}}e^{-f}$.
 Then $$f=4\tau V+C, \textrm{ where } C=\ln\left( (4\pi\tau)^{-\frac{n}{2}}\int_M e^{-4\tau V} \, d\upsilon\right).$$ 
  Therefore,  $\mu(g,\tau)$ is attained by $4\tau V+C$.
\vskip 2mm
  Conversely, if $\mu(g,\tau)$ is attained by $4\tau V+C$, we set $f=4\tau V+C$ and $\phi^2=(4\pi\tau)^{-\frac{n}{2}}e^{-f}$. 
  Then \begin{equation}\label{mu2}
  \begin{aligned}
      \int_M\Big(\tau\big|\nabla\ln \phi^2\big|^2&+\tau R - \ln \phi^2\Big)\phi^2\,d\upsilon\\
&\qquad -n\left(1+\frac{1}{2}\ln(4\pi\tau)\right)=\mu(g,\tau).
  \end{aligned}
   \end{equation}
 Since
  $$\phi^2=\frac{e^{-4\tau V}}{\int_M e^{-4\tau V} \, d\upsilon},$$ 
  by Lemma \ref{gibbs}, we have
  \begin{equation}\label{teo3p1}
\int_M \phi^2 \ln \phi^2 \, d\upsilon+4\tau \int_M V\phi^2 \, d\upsilon=-\ln\int_M e^{-4\tau V}\, dv.
\end{equation}
Hence by \eqref{mu2} and \eqref{teo3p1}, we obtain
\begin{equation}\label{teo3p2}
     \begin{aligned}
&\int_M \left(\frac{1}{4}|\nabla \ln \phi^2|^2+\frac{R}{4}+V\right) \phi^2 \, d\upsilon \\
&= -\frac{1}{4\tau}\ln\int_M e^{-4\tau V}\, dv
+\frac{\mu(g,\tau)}{4\tau}+\frac{n}{4\tau}\left(1+\frac{1}{2}\ln(4\pi\tau)\right).
\end{aligned}
\end{equation}
 Thus $\lambda_0(\mathcal{L})$ is attained by $\phi$. The proof is complete.
\end{proof}
\vskip 2mm
We now prove Theorem \ref{mainteo3} as a consequence of Theorem \ref{mainteo1}.

\begin{proof}[Proof of Theorem \ref{mainteo3}]
As shown in the proof of Lemma \ref{r1},  there exists a unitary isomorphism $U:L^2(M)\rightarrow L^2(M,\mu)$ such that
\begin{equation}\label{e0t22}-\Delta_f+V=U\left(-\Delta+\frac{R}{4}+\tilde{V}\right)U^{-1}\end{equation}
\noindent where 
$$\tilde{V}=V-\frac{R_f}{4}.$$
Then, by \eqref{rf},
\begin{equation}\label{e00t22}\tilde{V}=V+\frac{f}{4\tau}-\frac{(n+\mu_s)}{4\tau}.\end{equation}
Hence,
\begin{equation}\label{e1t22}\int_Me^{-4\tau \tilde{V}}d\upsilon=e^{n+\mu_s}(4\pi\tau)^{\frac{n}{2}}\int_M e^{-4\tau V} d\mu.\end{equation}
In particular, by $ e^{-4\tau V}\in L^1(M,\mu)$,  $ e^{-4\tau \tilde{V}}\in L^1(M)$.
\vskip 2mm
 Note that, by Remark \ref{rCZ1}, if $M$ is noncompact, then  $f$ is confining. Consequently,  $\tilde{V}$ is also confining. 
 Therefore, by Theorem \ref{mainteo1} and (\ref{e0t22}), the operator  $\mathcal{L}_f$ on $L^2(M,\mu)$, has discrete spectrum and 
\begin{equation}\label{e2t22}\lambda_0(\mathcal{L}_f)\geq -\frac{1}{4\tau}\ln\int_M e^{-4\tau \tilde{V}}\, d\upsilon + \frac{\mu_s}{4\tau}+\frac{n}{4\tau}\left(1+\frac{1}{2}\ln(4\pi\tau)\right).\end{equation}
Using (\ref{e1t22}), we obtain
\begin{equation}\label{e3t22}\lambda_0(\mathcal{L}_f)\geq -\frac{1}{4\tau}\ln\int_M e^{-4\tau V} \, d\mu.\end{equation}

Observe  that equality holds in \eqref{e3t22} if and only if equality holds in \eqref{e2t22}. By Theorem \ref{mainteo1}, this is equivalent to the fact that $4\tau\tilde{V}$ is a potential function of the Ricci shrinker; equivalently, $\tilde{V}\in C^{\infty}(M)$ and
$$\nabla^2(4\tau\tilde{V}-f)=0.$$
Hence, by (\ref{e00t22}), we obtain $V\in C^{\infty}(M)$ with $\nabla^2V=0$.
\end{proof}

\vskip 2mm

In view of Remark \ref{remav}, Theorem \ref{mainteo3} admits the following formulation for compact Ricci shrinkers.

\begin{theorem}\label{t22v2}
Let $(M,g,f)$ be an $n$-dimensional compact   Ricci shrinker such that the measure $\mu=(4\pi\tau)^{-\frac{n}{2}}e^{-f}d\upsilon$ is a probability measure, and let $V$  satisfy $e^{-4\tau V}\in L^1(M,\mu)$. Then,  $$\mathcal{L}_f=-\Delta_f+V$$ on $L^2(M,\mu)$ has discrete spectrum, and its lowest eigenvalue satisfies
    $$\lambda_0(\mathcal{L}_f)\geq -\frac{1}{4\tau}\ln\int_M e^{-4\tau V}\, d\mu.$$
Moreover, equality holds if and only if $V$ is a constant function.
\end{theorem}

Our final result provides a lower bound for the bottom of spectrum of $\mathcal{L}_f$  in terms of the  $\mu$-functional on the smooth metric measure space $(M,g,f)$.

\begin{theorem}\label{mainteo4} Let $(M,g,f)$ be an $n$-dimensional complete smooth metric measure space, and let $V$ satisfy $e^{-4\tau V}\in L^1(M,d\mu)$ for some $\tau>0$, where   $d\mu=(4\pi\tau)^{-\frac{n}{2}}e^{-f}d\upsilon$ is the weighted measure. Then, the bottom of the spectrum of  
$$\mathcal{L}_f=-\Delta_f+V$$ on $L^2(M,d\mu)$ satisfies
$$\lambda_0(\mathcal{L}_f)\geq -\frac{1}{4\tau}\ln \int_M e^{-4\tau V}d\mu+\mu(g,\tau)+ \frac{1}{4\tau}\inf_u\left\{ \int_M(n-f-\tau R_f) u^2 \, d\mu\right\},$$
where  $R_f=R+|\nabla f|^2+2\Delta_ff$ and  the infimum is taken over all functions $u\in C^{\infty}_c(M)$ with $\int_M u^2 \, d\mu=1$.
\end{theorem}
\begin{proof}
The bottom of spectrum of $\mathcal{L}_f$ satisfies
\[\lambda_0(\mathcal{L}_f)=\inf_{u}\left\{\int_M \left(|\nabla u|^2 + Vu^2\right) \, d\mu\right\},
\]
where the infimum is taken over all  functions $u\in C^{\infty}_c(M)$ satisfying $\int_M u^2 \, d\mu=1$.
\vskip 2mm
Let $u\in C^{\infty}_c(M)$ with $\int_M u^2 \, d\mu=1$ and  define $\phi=u(4\pi\tau)^{-\frac{n}{4}}e^{-\frac{f}{2}}.$
Then $$\int_M \phi^2 \, d\upsilon=\int_M u^2 \, d\mu=1,$$
\begin{equation}
\begin{aligned}
    &\int_M\big|\nabla \phi|^2\,d\upsilon \\
    & =  \int_M\left(|\nabla u|^2+\frac14|\nabla f|^2u^2-u\langle \nabla u,\nabla f\rangle\right)(4\pi\tau)^{-\frac{n}{2}}e^{-f}\,d\upsilon  \\
& =  \int_M\left(|\nabla u|^2+\frac14|\nabla f|^2u^2-\frac12\langle \nabla u^2,\nabla f\rangle\right) \, d\mu\\
& =  \int_M\left(|\nabla u|^2+\frac14|\nabla f|^2u^2+\frac12(\Delta_ff)u^2\right) \, d\mu,
\end{aligned}
\end{equation}
and
\begin{align}
  \int_M\phi^2\ln \phi^2d\upsilon&=\int_Mu^2\left( \ln u^2 -\frac{n}{2}\ln(4\pi\tau)-f\right)d\mu\nonumber\\
  &=\int_M\left(u^2\ln u^2 -fu^2\right)d\mu-\frac{n}{2}\ln(4\pi\tau).  
\end{align}
It follows that
\begin{equation}\label{mu3}
    \begin{aligned}
    &\int_M\Big(4\tau|\nabla \phi|^2+\tau R\phi^2 -\phi^2 \ln \phi^2\Big)\,d\upsilon-n\left(1+\frac{1}{2}\ln(4\pi\tau)\right) \\
    &= \int_M\Big(4\tau|\nabla u|^2+\tau|\nabla f|^2u^2+2\tau(\Delta_ff)u^2+\tau Ru^2 \Big)\,d\mu\\
    &\qquad \qquad-\int_M\left(u^2\ln u^2 -fu^2\right)d\mu-n\\
    &=\int_M\Big(4\tau|\nabla u|^2+\tau R_f u^2 -u^2\ln u^2 +fu^2\Big)\,d\mu-n
\end{aligned}
\end{equation}
By \eqref{perelman}, \eqref{mu3} implies that
\begin{equation}\label{eq1}\int_M\Big(4\tau|\nabla u|^2+\tau R_f u^2 -u^2\ln u^2 +fu^2\Big)\,d\mu-n\;\ge\; \mu(g,\tau).\end{equation}
It follows that
$$
4\tau\int_M|\nabla u|^2 \, d\mu \geq \int_M u^2\ln u^2  \, d\mu  +\int_M(n-f-\tau R_f) u^2 \, d\mu+\mu(g,\tau).$$
Consequently,
\begin{eqnarray*}
4\tau\int_M\left(|\nabla u|^2+V\right) u^2 \, d\mu & \geq & \int_M u^2\ln u^2 \, d\mu +4\tau\int_M Vu^2 \, d\mu\\
& & + \int_M(n-f-\tau R_f) u^2 \, d\mu+\mu(g,\tau).
\end{eqnarray*}
This implies that
\begin{eqnarray*} 4\tau \lambda_0(\mathcal{L}_f) &\geq & \inf_{u}\left\{\int_M u^2\ln u^2 \, d\mu +4\tau\int_M Vu^2 \, d\mu\right\}\\
& & +\inf_u\left\{ \int_M(n-f-\tau R_f) u^2 \, d\mu\right\}+\mu(g,\tau).\end{eqnarray*}
Hence, by Lemma \ref{gibbs}, we obtain
$$\lambda_0(\mathcal{L}_f)\geq -\frac{1}{4\tau}\ln \int_M e^{-4\tau V}d\mu+\mu(g,\tau)+\frac{1}{4\tau}\inf_u\left\{ \int_M(n-f-\tau R_f) u^2 \, d\mu\right\},$$
where the infimum is taken over all functions $u\in C^{\infty}_c(M)$ such that $\int_M u^2 \, d\mu=1$.
\end{proof}

\end{document}